\documentclass[a4paper,12pt]{amsart}
\usepackage{amsmath}
\usepackage{amssymb}
\usepackage{xcolor}
\usepackage{amsthm}
\usepackage{amsfonts}
\usepackage[francais, english]{babel}

\DeclareMathSymbol{\subsetneqq}{\mathbin}{AMSb}{36}

\newcommand{\R}{\mathbb{R}}

\newcommand{\N}{\mathbb{N}}

\newcommand{\C}{\mathbb{C}}

\newcommand{\beq}{\begin{eqnarray}}
\newcommand{\eeq}{\end{eqnarray}}
\newcommand{\bq}{\begin{equation}}
\newcommand{\eq}{\end{equation}}
\newcommand{\beqn}{\begin{eqnarray*}}
\newcommand{\eeqn}{\end{eqnarray*}}
\newcommand{\bex}{\begin{exo}}
\newcommand{\eex}{\end{exo}}
\newcommand{\ben}{\begin{enumerate}}
\newcommand{\een}{\end{enumerate}}

\newtheorem{th1}{{\bf Theorem}}[section]
\newtheorem{thm}[th1]{{\bf Theorem}}
\newtheorem{lem}[th1]{{\bf Lemma}}
\newtheorem{prop}[th1]{{\bf Proposition}}

\newtheorem{rem}[th1]{\bf Remark}

\newtheorem{defi}[th1]{\bf Definition}

\author[T. Saanouni]{T. Saanouni}
\address{University Tunis El Manar,
Faculty of Sciences of Tunis, LR03ES04 partial differential equations and applications, 2092 Tunis, Tunisia.}
\email{\sl Tarek.saanouni@ipeiem.rnu.tn}


\title[Coupled NLS]{A note on coupled focusing nonlinear Schr\"odinger equations}

\date{\today}
\begin{document}
\begin{abstract}
Some focusing coupled Schr\"odinger equations are investigated. First, existence of ground state is obtained. Second, global and non global existence of solutions are discussed via potential-well method. Finally, strong instability of standing waves is established. 
\end{abstract}
\maketitle
\tableofcontents
\vspace{ 1\baselineskip}
\renewcommand{\theequation}{\thesection.\arabic{equation}}
\section{Introduction}
Consider the Cauchy problem for a focusing Schr\"odinger system with power-type nonlinearities
\begin{equation}
\left\{
\begin{array}{ll}
i\dot u_j +\Delta  u_j=  -\displaystyle\sum_{k=1}^{m}a_{jk}|u_k|^p|u_j|^{p-2}u_j ;\\
u_j(0,.)= \psi_{j},
\label{S}
\end{array}
\right.
\end{equation}
where $u_j: \R^{N} \times \R \to \C$, for $j\in[1,m]$ and $a_{jk} =a_{kj}$ are positive real numbers.\\
The m-component coupled nonlinear Schr\"odinger system with power-type nonlinearities
denoted $(CNLS)_p$ arises in many physical problems such as nonlinear optics and Bose-Einstein condensates. It models physical systems in which the field has more than one component. In nonlinear optics \cite{ak} $u_j$ denotes the $j^{th}$ component of the beam in Kerr-like photo-refractive media. The coupling constant $a_{jk}$ acts to the interaction between the $j^{th}$ and the $k^{th}$ components of the beam. $(CNLS)_p$ arises also in the Hartree-Fock theory for a two component Bose-Einstein condensate. 
 Readers are referred, for instance, to \cite{Hasegawa, Zakharov} for the derivation and applications of this system.\\

A solution ${\bf u}:= (u_1,...,u_m)$ to \eqref{S} formally satisfies respectively conservation of the mass and the energy
\begin{gather*}
M(u_j):= \displaystyle\int_{\R^N}|u_j(t,x)|^2\,dx = M(\psi_{j});\\
E({\bf u}(t)):= \frac{1}{2}\displaystyle \sum_{j=1}^{m}\displaystyle\int_{\R^N}|\nabla u_j(t,x)|^2\,dx - \frac{1}{2p}\displaystyle \sum_{j,k=1}^{m}a_{jk}\displaystyle \int_{\R^N} |u_j(t,x)u_k(t,x)|^p\,dx = E({\bf u}(0)).
\end{gather*}
Before going further, let us recall some historic facts about this problem. 
For the one component Schr\"odinger equation, the model case given by a pure power nonlinearity is of particular interest. The question of well-posedness in the energy space $H^1$ was widely investigated. We denote for $p>1$ the Schr\"odinger problem
$$(NLS)_p\quad i\dot  u+\Delta u\pm u|u|^{p-1}=0,\quad u:{\mathbb R}\times{\mathbb R}^N\rightarrow{\mathbb C}.$$
This equation satisfies a scaling invariance. Indeed, if $u$ is a solution to $(NLS)_p$ with data $u_0$, then
$ u_\lambda:=\lambda^{\frac2{p-1}}u(\lambda^2\, .\,,\lambda\, .\,)$
is a solution to $(NLS)_p$ with data $\lambda^{\frac2{p-1}}u_0(\lambda\,.\,).$
For $s_c:=\frac N2-\frac2{p-1}$, the space $\dot H^{s_c}$ whose norm is invariant under the dilatation $u\mapsto u_{\lambda}$ is relevant in this theory. When $s_c=1$, which is the energy critical case, the critical power is $p_c:=\frac{N+2}{N-2}$, $N\geq 3$. \\
Local well-posedness holds in the energy critical case \cite{Cas.F} and the local existence interval does not depend only on $\|u_0\|_{H^1}$. Then, an iteration of the local well-posedness theory fails to prove global existence. But using ideas of Bourgain \cite{J.B1,J.B2} and a new interaction Morawetz inequality \cite{Col.K} the energy critical case of $(NLS)_p$ is now completely resolved \cite{V,RV07}. Finite energy initial data evolve into global solution $u$ with finite spacetime size $\|u\|_{L_{t,x}^{\frac{2(2+N)}{N-2}}}<\infty$ and scatter.\\

In two space dimensions, similar results about global well-posedness and scattering of the Schr\"odinger equation with exponential nonlinearity exist \cite{T,T1,T2,T3}.\\

Intensive work has been done in the last few years about coupled Schr\"odinger systems \cite{ntds,w,mz}. These works have been mainly on 2-systems or with small couplings. Moreover, most works treat the focusing case by considering the stationary associated problem \cite{AC2,xs,hs,AC3,AC4}. Despite the partial progress made so far, many difficult questions remain open and little is known about m-systems for $m\geq 3$.\\

Recently, the defocusing problem associated to the Schr\"odinger system \eqref{S} was investigated in \cite{saa}, where global existence and scattering were obtained.\\

The purpose of this manuscript is two-fold. First, by obtaining existence of a ground state, global and non global well-posedness of the system \eqref{S} is discussed via potential well method. Second, using variational methods, we study the stability of standing waves.\\

 The rest of this paper is organized as follows. The next section contains the main results and some technical tools needed in the sequel. The goal of the third section is to study the stationary problem associated to \eqref{S}. In section four, global and non global existence of solutions is discussed via the potential-well theory. The last section is devoted to study the stability of standing waves.\\

We end this section with some definitions. Let the product space 
$$H:={H^1({\R^N})\times...\times H^1({\R^N})}=[H^1({\R^N})]^m$$
where $H^1(\R^N)$ is the usual Sobolev space endowed with the complete norm 
$$ \|u\|_{H^1(\R^N)} := \Big(\|u\|_{L^2(\R^N)}^2 + \|\nabla u\|_{L^2(\R^N)}^2\Big)^\frac12.$$
We denote the real numbers 
 $$p_*:=1+\frac2N\quad\mbox{ and }\quad p^*:=\left\{
\begin{array}{ll}
\frac{N}{N-2}\quad\mbox{if}\quad N>2;\\
\infty\quad\mbox{if}\quad N=2.
\end{array}
\right.$$
We mention that $C$ will denote a
constant which may vary from line to line and if $A$ and $B$ are non negative real numbers, $A\lesssim B$  means that $A\leq CB$. For $1\leq r\leq\infty$, we denote the Lebesgue space $L^r:=L^r({\mathbb R}^N)$ with the usual norm $\|\,.\,\|_r:=\|\,.\,\|_{L^r}$ and $\|\,.\,\|:=\|\,.\,\|_2$.
\, For simplicity, we denote the usual Sobolev Space $W^{s,p}:=W^{s,p}({\mathbb R}^N)$ and $H^s:=W^{s,2}$. If $X$ is an abstract space $C_T(X):=C([0,T],X)$ stands for the set of continuous functions valued in $X$ and $X_{rd}$ is the set of radial elements in $X$, moreover for an eventual solution to \eqref{S}, we denote $T^*>0$ it's lifespan.
\section{Main results and background}
In what follows, we give the main results and some estimates needed in the sequel.
 For ${\bf u} :=(u_1,...,u_m)\in H$, we define the action
$$S({\bf u}):= \frac{1}{2}\displaystyle\sum_{j=1}^m\| u_j\|_{H^1}^2-\frac{1}{2p}\displaystyle\sum_{j,k=1}^m a_{jk} \displaystyle\int_{\R^N}|u_j u_k|^{p}\,dx.$$
If $ \alpha, \,\beta\in \R,$ we call constraint
{\small$${2K_{\alpha,\beta}({\bf u}):=\displaystyle\sum_{j=1}^m\big((2\alpha + (N -2)\beta) \|\nabla u_j\|^2 + (2\alpha + N \beta) \| u_j\|^2   \big) - \frac{1}{p}\displaystyle\sum_{j,k=1}^m a_{jk}\displaystyle\int_{\R^N}(2p\alpha + N \beta)|u_j u_k|^{p}\,dx.}$$}
\begin{defi}
We say that $\Psi:=(\psi_1,...,\psi_m)$ is a ground state solution to \eqref{S} if
\begin{equation}\label{E}
\Delta  \psi_j - \psi_j + \displaystyle\sum_{k=1}^m a_{jk}|\psi_k|^p|\psi_j|^{p - 2}\psi_j=0,\quad 0\neq \Psi\in H_{rd}
\end{equation}
and it minimizes the problem
\begin{equation}\label{M}
m_{\alpha,\beta}:= \inf_{0\neq {\bf u}\in H}\{ S({\bf u})\quad\mbox{s.\,t}\quad K_{\alpha,\beta}({\bf u}) = 0\}.
\end{equation}
Moreover, in such a case $\Psi$ is called vector ground state if each component is nonzero.
\end{defi}
\begin{rem}
If $\Psi\in H$ is a solution to \eqref{E}, then $e^{it}\Psi$ is a global solution of \eqref{S} said standing wave.
\end{rem}
\subsection{Main results}
First, the existence of a ground state solution to \eqref{S} is claimed.
\begin{thm}\label{t1}
Take $N\geq2$, $p_*<p<p^*$ and two real numbers $(0, 0)\, \neq \, (\alpha,\beta)\in \R_{+}^2\cup\{1,-\frac2N\}.$ Then
\begin{enumerate}
\item $ m:=m_{\alpha,\beta}$ is nonzero and independent of $(\alpha,\beta);$
\item there is a minimizer of \eqref{M}, which is some nontrivial solution to \eqref{E};
\item
if we make the following assumptions
$$a_{jj} = \mu_j\; \mbox{ and }\; a_{jk} = \mu \quad\mbox{for}\quad j\neq k\in[1,m]$$
then, at least two components of the minimizer are non zero if $\mu>0$ is large enough.
\end{enumerate}
\end{thm}
Second, using the potential well method \cite{ps}, we discuss global and non global existence of a solution to the focusing problem \eqref{S}. Define the sets
\begin{gather*}
A_{\alpha,\beta}^+:= \{ {\bf u}\in H \quad\mbox{s.\, t}\quad S({\bf u})<m\quad\mbox{and}\quad K_{\alpha,\beta}({\bf u})\geq 0\};\\
A_{\alpha,\beta}^-:= \{ {\bf u}\in H \quad\mbox{s.\, t}\quad S({\bf u})<m\quad\mbox{and}\quad K_{\alpha,\beta}({\bf u})< 0\}.
\end{gather*}
\begin{thm}\label{t2}
Take $2\leq N\leq 4$ and $ p_*< p< p^*.$ Let $\Psi \in H$ and ${\bf u}\in C_{T^*}(H)$ the maximal solution to \eqref{S}. 
\begin{enumerate}
\item
If there exist $(0,0)\neq (\alpha,\beta)\in \R_+^2\cup\{1,-\frac2N\} $ and $t_0\in  [0, T^*)$ such that 
${\bf u}(t_0)\in A_{\alpha,\beta}^+$, then ${\bf u}$ is global;
\item
if there exist $(0,0)\neq (\alpha,\beta)\in \R_+^2\cup\{1,-\frac2N\} $ and $t_0\in  [0, T^*)$ such that 
${\bf u}(t_0)\in A_{\alpha,\beta}^-$ and $x{\bf u}(t_0)\in L^2$, then ${\bf u}$ is non global;
\end{enumerate}
\end{thm}
\begin{rem}
The existence of a local solution to \eqref{S} was proved in \cite{saa}.
\end{rem}
The last result concerns instability by blow-up for standing waves of the Schr\"odinger problem \eqref{S}. Indeed, near ground state, there exist infinitely many data giving finite time blowing-up solutions to \eqref{S}.
\begin{thm}\label{t3}
Take $2\leq N\leq 4$ and $ p_*< p< p^*.$ Let $\Psi$ be a ground state solution to \eqref{E}. Then, for any $\varepsilon>0$, there exists ${\bf u}_0\in H$ such that $\|{\bf u}_0-\Psi\|_{H}<\varepsilon$ and the maximal solution to \eqref{S} with data ${\bf u}_0$ is not global.
\end{thm}
In what follows, we collect some intermediate estimates.
\subsection{Tools}
Any solution to \eqref{S} formally enjoys the so-called Virial identity \cite{ntds}.
\begin{prop}\label{vir}
Let ${\bf u}\in H$, a solution to \eqref{S} such that $x{\bf u}\in L^2$. Then,
\begin{equation}\label{vrl}
\frac{1}8\Big(\displaystyle\Sigma_{j=1}^m\|xu_j(t)\|_{L^2}^2\Big)''=\displaystyle\Sigma_{j=1}^m\|\nabla u_j\|_{L^2}^2-\frac{N(p-1)}{2p}\displaystyle\Sigma_{j,k=1}^m\int_{\R^N}|u_ju_k|^{p}\,dx.
\end{equation}
\end{prop}
The following Gagliardo-Nirenberg inequality \cite{gn} will be useful.
\begin{prop}\label{intrp}
For any $(u_1,..,u_m)\in H$, yields
\begin{equation}\label{Nirenberg}
\displaystyle \sum_{j,k=1}^{m}\displaystyle \int_{\R^N} |u_ju_k|^p\,dx \leq C \left(\displaystyle\sum_{j=1}^{m}\|\nabla u_j\|^2\right)^{\frac{(p-1)N}2}\left(\displaystyle\sum_{j=1}^{m}\|u_j\|^2\right)^{\frac{N-p(N -2)}2}.\end{equation}
\end{prop}
Let us list some Sobolev embeddings \cite{Adams,Lions}.
\begin{prop}\label{injection}
The first injection is continuous and the last one is compact.
\begin{enumerate}
\item $ W^{s,p}(\R^N)\hookrightarrow L^q(\R^N)$ whenever
$1<p<q<\infty, \quad s>0\quad \mbox{and}\quad \frac{1}{p}\leq \frac{1}{q} + \frac {s}{N};$
\item
for $2<p< 2 p^*,$ 
\begin{equation}\label{radial} H_{rd}^1(\R^N)\hookrightarrow\hookrightarrow L^p(\R^N).\end{equation}
\end{enumerate}
\end{prop}
Finally, recall the so-called generalized Pohozaev identity \cite{sl1}.
\begin{prop}
$\Psi \in H$ is a solution to \eqref{E} if and only if $S'(\Psi)=0.$ Moreover, in such a case 
$$K_{\alpha,\beta}(\Psi)=0,\quad\mbox{for any}\quad (\alpha,\beta)\in\R^2.$$
\end{prop}
\section{The stationary problem }
The goal of this section is to prove that the elliptic problem \eqref{E} has a ground state solution which is a vector one in some cases. Let us start with some notations. For ${\bf u} :=(u_1,...,u_m)\in H$ and $\lambda,\, \alpha, \,\beta\in \R,$ we introduce the scaling
$$( u_j^\lambda)^{\alpha,\beta}:= e^{\alpha\lambda}u_j(e^{-\beta \lambda}.)$$
and the differential operator
$$ \pounds_{\alpha,\beta}:H^1\to H^1,\quad u_j\mapsto \partial_\lambda((u_j^\lambda)^{\alpha,\beta})_{|\lambda=0}.$$
We extend the previous operator as follows, if $A:H^1(\R^N)\to \R,$ then 
$$\pounds_{\alpha,\beta}A(u_j):= \partial_\lambda (A((u_j^\lambda)^{\alpha,\beta}))_{|\lambda=0}.$$
Denote also the constraint
\begin{eqnarray*}
K_{\alpha,\beta}({\bf u})&:= &\partial_\lambda\big(S(({\bf u}^\lambda)^{\alpha,\beta})\big )_{|\lambda = 0}\\
&=& \frac{1}{2}\displaystyle\sum_{j=1}^m\Big((2\alpha + (N -2)\beta) \|\nabla u_j\|^2 + (2\alpha + N \beta) \| u_j\|^2   \Big) \\&- &\frac{1}{2p}\displaystyle\sum_{j,k=1}^m a_{jk}\displaystyle\int_{\R^N}(2p\alpha + N \beta)|u_j u_k|^{p}\,dx\\
&:=&\frac{1}{2}\displaystyle\sum_{j=1}^m K_{\alpha,\beta}^Q(u_j) - \frac{1}{2p}\displaystyle\sum_{j,k=1}^m a_{jk}\displaystyle\int_{\R^N}(2p\alpha + N \beta)|u_j u_k|^{p}\,dx.
\end{eqnarray*}
Finally, we introduce the quantity
\begin{eqnarray*}
H_{\alpha,\beta}({\bf u})
&:=& S({\bf u}) - \frac{1}{2\alpha + N\beta}K_{\alpha,\beta}({\bf u})\\
&=& \frac{1}{2\alpha +N\beta }\Big[\displaystyle\sum_{j=1}^m 2\beta \|\nabla u_j\|^2  + \alpha (1-\frac{1}{p})\displaystyle\sum_{j,k=1}^m a_{jk}\displaystyle\int_{\R^N}|u_j u_k|^{p}\,dx\Big].
\end{eqnarray*}
\subsection{Existence of ground state}
Now, we prove Theorem \ref{t1} about existence of a ground state solution to the stationary problem \eqref{E}.
\begin{rem} 
\begin{enumerate}
\item[(i)] The proof of the Theorem \ref{t1} is based on several lemmas;
\item[(ii)]we write, for easy notation, $u_j^\lambda:= (u_j^\lambda)^{\alpha,\beta},\, K:= K_{\alpha,\beta},\, K^Q:= K_{\alpha,\beta}^Q,\, \pounds:= \pounds_{\alpha, \beta}\, \mbox{and}\, H:= H_{\alpha,\beta}.$
\end{enumerate}
\end{rem}
\begin{lem} Let $(\alpha,\beta)\in \R_+^2.$ Then
\begin{enumerate}
\item $\min \big(\pounds H({\bf u}), H({\bf u})\big)\geq 0$ for all $0 \neq {\bf u} \in H;$
\item $\lambda \mapsto H({\bf u}^\lambda)$ is increasing.
\end{enumerate}
\end{lem}
\begin{proof}
We have
$$H({\bf u})\geq \frac{2\beta}{2\alpha + N\beta}\|\nabla u_j\|^2\geq 0.$$
 Moreover, with a direct computation
{\small\begin{eqnarray*}
\pounds H({\bf u}) 
&= &\pounds \big(1 - \frac{\pounds}{2\alpha + N\beta}\big)S({\bf u})\\
&=& \frac{-1}{2\alpha + N \beta}\big(\pounds - (2\alpha + (N -2)\beta)\big)\big(\pounds - (2\alpha + N \beta)\big)S({\bf u}) +( 2\alpha + (N -2)\beta)\big(1 - \frac{\pounds}{2\alpha + N\beta}\big)S({\bf u})\\
&=& \frac{-1}{2\alpha + N \beta}\big(\pounds - (2\alpha + (N -2)\beta)\big)\big(\pounds - (2\alpha + N \beta)\big)S({\bf u}) +( 2\alpha + (N - 2)\beta)H({\bf u}).
\end{eqnarray*}}
Since $\big (\pounds - (2\alpha + (N -2)\beta)\big) \|\nabla u_j\|^2 = \big(\pounds - (2\alpha + N\beta)\big)\|u_j\|^2 = 0,$ we have
$\big (\pounds - (2\alpha + (N - 2)\beta)\big)  \big(\pounds - (2\alpha + N\beta)\big)\|u_j\|_{H^1}^2 =0$ and
\begin{eqnarray*}
\pounds H({\bf u})&\geq&\frac{-1}{2\alpha + N \beta}\big(\pounds - (2\alpha + (N -2)\beta)\big)\big(\pounds - (2\alpha + N \beta)\big)\Big(\frac{-1}{2p}\displaystyle\sum_{j,k=1}^m a_{jk}\displaystyle\int_{\R^N}|u_ju_k|^p\,dx\Big)\\
&\geq& \frac{1}{2p}\frac{2\alpha (p - 1)}{2\alpha + N\beta}\big( 2\alpha(p - 1) +2\beta \big)\displaystyle\sum_{j,k=1}^m a_{jk}\displaystyle\int_{\R^N}|u_ju_k|^p\,dx\geq 0.
\end{eqnarray*}
The last point is a consequence of the equality $ \partial_\lambda H({\bf u}^\lambda) = \pounds H({\bf u}^\lambda).$ 
\end{proof}
The next intermediate result is the following.
\begin{lem} \label{K>0} Let $(\alpha,\beta)\in \R^2$ satisfying $2\alpha+(N-2)\beta>0$, $2\alpha+N\beta\geq0$ and $0\, \neq\, (u_1^n,...,u_m^n)$ be a bounded sequence of $H$ such that
$$ \lim_n\big(\displaystyle\sum_{j=1}^m K^Q(u_j^n)\big) =0.$$
Then there exists $n_0\in \N$ such that $K(u_1^n,...,u_m^n)>0$ for all $n\geq n_0.$
\end{lem}
\begin{proof}
We have,
$$K(u_1^n,...,u_m^n) =\frac{1}{2}\displaystyle\sum_{j=1}^m K^Q(u_j^n)  - \frac{(2p\alpha + N \beta)}{2p}\displaystyle\sum_{j,k=1 }^m a_{jk}\displaystyle\int_{\R^N}|u_j^n u_k^n|^{p}\,dx.$$
Using Proposition \eqref{intrp}, since $p_*<p<p^*$, $\min\{(2\alpha+(N-2)\beta),2\alpha+N\beta\}>0$ and
$$K^Q(u_j^n) = \Big((2\alpha + (N -2)\beta) \|\nabla u_j^n\|^2 + (2\alpha + N \beta) \| u_j^n\|^2   \Big) \rightarrow 0,$$
yields
$$   \displaystyle \sum_{j,k=1}^{m}a_{jk}\displaystyle \int_{\R^N} |u_j^nu_k^n|^p = o\left( \displaystyle\sum_{j=1}^m \|\nabla u_j^n\|^2\right) = o\left(\displaystyle\sum_{j=1}^m K^Q(u_j^n)  \right) .   $$
Thus
$$K(u_1^n,...,u_m^n) \simeq\frac{1}{2}\displaystyle\sum_{j=1}^m K^Q(u_j^n)\geq 0 . $$
\end{proof}
We read an auxiliary result.
\begin{lem}\label{Lemma} 
Let $(0,0)\neq(\alpha, \beta)\in\R_+^2$ satisfying $(N,\alpha)\neq (2,0)$. Then
$$m_{\alpha,\beta}  = \inf_{0\neq{\bf u}\in H}\big\{H({\bf u})\quad\mbox{s.\, t}\quad K({\bf u})\leq 0 \big\}.$$
\end{lem}
\begin{proof} 
Denoting by $a$ the right hand side of the previous equality, 
 it is sufficient to prove that $m_{\alpha,\beta}\leq a.$ Take ${\bf u} \in H$ such that $K({\bf u})<0.$ Because $\displaystyle\lim_{\lambda\rightarrow -\infty}K^Q({\bf u}^\lambda)=0,$ by the previous Lemma, there exists some $\lambda<0$ such that $ K({\bf u}^\lambda)>0.$ With a continuity argument there exists $\lambda_0\leq0$ such that $K ({\bf u}^{\lambda_0}) = 0,$ then since $\lambda\mapsto H({\bf u}^\lambda)$ is increasing, we get
$$ m_{\alpha, \beta}\leq H({\bf u}^{\lambda_0}) \leq H({\bf u}).$$
This closes the proof.
\end{proof}{}
Let prepare the proof of the last part of the Theorem \ref{t1}. Here and hereafter, for $\lambda>0$ and ${\bf u}:=(u_1,..,u_m)\in H$ we denote ${\bf u}_{\lambda}:=\lambda^{\frac{N}2}{\bf u}(\lambda.)$ and 
\begin{gather*}
Q_j({\bf u}):=\frac2N\|\nabla u_j\|^2-(1-\frac1p)\displaystyle\displaystyle\Sigma_{k=1}^ma_{j,k}\int_{\R^N}|u_ju_k|^p\,dx;\\
S_j({\bf u}):=\frac12\|u_j\|_{H^1}^2-\frac1{2p}\displaystyle\displaystyle\Sigma_{k=1}^ma_{j,k}\int_{\R^N}|u_ju_k|^p\,dx.
\end{gather*}
Note that $\displaystyle\displaystyle\Sigma_{j=1}^mQ_j=K_{1,\frac{-2}N}$ and $\displaystyle\displaystyle\Sigma_{j=1}^mS_j=S$.
\begin{lem}\label{cle}
 Let ${\bf u}\in H$ such that $\displaystyle\displaystyle\Sigma_{j=1}^mQ_j({\bf u})\leq 0$. Then, there exists $\lambda_0\leq 1$ such that
\begin{enumerate}
\item
$\displaystyle\displaystyle\Sigma_{j=1}^mQ_j({\bf u}_{\lambda_0})=0$;
\item
$\lambda_0=1$ if and only if $\displaystyle\displaystyle\Sigma_{j=1}^mQ_j({\bf u})=0$;
\item
$\frac{\partial}{\partial\lambda}S({\bf u}_{\lambda})>0$ for $\lambda\in (0,\lambda_0)$ and $\frac{\partial}{\partial\lambda}S({\bf u}_{\lambda})<0$ for $\lambda\in (\lambda_0,\infty)$;
\item
$\lambda\rightarrow S_j({\bf u}_{\lambda})$ is concave on $(\lambda_0,\infty)$;
\item
$\frac{\partial}{\partial\lambda}S_j({\bf u}_{\lambda})=\frac N{2\lambda}Q_j({\bf u}_{\lambda}).$
\end{enumerate}
\end{lem}
\begin{proof}
We have
$$Q_j({\bf u}_\lambda)=\frac{2\lambda^2}N\|\nabla u_j\|^2-(1-\frac1p)\lambda^{N(p-1)}\displaystyle\displaystyle\Sigma_{k=1}^ma_{j,k}\int_{\R^N}|u_ju_k|^p\,dx.$$
Moreover, with previous computations
\begin{eqnarray*}
\partial_{\lambda}S_j({\bf u}_{\lambda})
&=&\lambda\|\nabla v_j\|^2-\frac N2(1-\frac1p)\lambda^{N(p-1)-1}\displaystyle\displaystyle\Sigma_{k=1}^ma_{j,k}\int_{\R^N}|u_ju_k|^p\,dx\\
&=&\frac N{2\lambda}Q_j({\bf u}_{\lambda})
\end{eqnarray*}
which proves $(5)$. Now
\begin{eqnarray*}
Q_j(v_\lambda)
&=&\frac{2\lambda^2}N\|\nabla u_j\|^2-(1-\frac1p)\lambda^{N(p-1)}\displaystyle\displaystyle\Sigma_{k=1}^ma_{j,k}\int_{\R^N}|u_ju_k|^p\,dx\\
&=&\frac{2\lambda^2}N\Big[\|\nabla u_j\|^2-\frac N2(1-\frac1p)\lambda^{N(p-1)}\displaystyle\displaystyle\Sigma_{k=1}^ma_{j,k}\int_{\R^N}|u_ju_k|^p\,dx\Big].
\end{eqnarray*}
A monotony argument closes the proof of $(1),(2)$ and $(3)$. For $(4)$, it is sufficient to compute using $(3)$.
\end{proof}
In the case $(\alpha,\beta)=(1,-\frac2N)$, we will use $T:=S-\frac4NK_{1,-\frac2N}$ rather then $H_{\alpha,\beta}$ which is no longer defined.
\begin{lem}\label{dec}
For ${\bf u}\in H$, the following real function is increasing on $\R_+$,
$$\lambda\mapsto T(\lambda{\bf u}).$$
\end{lem}
\begin{proof}
Denoting ${\bf u}:=(u_1,..,u_m)\in H$, we compute
{\small\begin{gather*}
T(\lambda{\bf u})=\frac{\lambda^2}2\displaystyle\displaystyle\Sigma_{j=1}^m\Big(\|u_j\|^2+\frac{N\lambda^{2p-2}}4(1-\frac1p-\frac2{Np})\displaystyle\displaystyle\Sigma_{k=1}^m\int_{\R^N}|u_ju_k|^p\,dx\Big);\\
\partial_\lambda T(\lambda{\bf u})={\lambda}\displaystyle\displaystyle\Sigma_{j=1}^m\Big(\|u_j\|^2+\frac{N\lambda^{2p-3}}2(p-1-\frac2{N})\displaystyle\displaystyle\Sigma_{k=1}^m\int_{\R^N}|u_ju_k|^p\,dx\Big);
\end{gather*}}
The proof is ended because $p>p_*$.
\end{proof}
As previously, we can express the minimizing number $m_{1,-\frac2N}$ with a negative constraint.
\begin{prop}\label{enf}
We have
$$m_{1,-\frac2N}=\displaystyle\inf_{0\neq{\bf u}\in H}\{T({\bf u}),\quad K_{1,-\frac2N}({\bf u})\leq0\}.$$
\end{prop}
\begin{proof}
Letting $m_1$ the right hand side, it is sufficient to prove that $m_{1,-\frac2N}\leq m_1$. Take ${\bf u}\in H$ such that $K_{1,-\frac2N}({\bf u})<0$ then by Lemma \ref{cle} and the facts that $\lambda\mapsto T(\lambda{\bf u})$ is increasing, there exists $\lambda\in(0,1)$ such that $K_{1,-\frac2N}(\lambda{\bf u})=0$ and $m_{1,-\frac2N}\leq T(\lambda{\bf u})\leq T({\bf u})$. The proof is closed.
\end{proof}
{\bf Proof of theorem \ref{t1}}\\
\underline{First case $(N,\alpha) \neq (2,0).$}\\
Let $(\phi_n):=(\phi_1^n,...,\phi_m^n) $ be a minimizing sequence, namely
\begin{equation} \label{suite}0\neq (\phi_n) \in H,\quad K(\phi_n) = 0\quad \mbox{and}\quad \lim_n H(\phi_n) = \lim_n S(\phi_n) = m.\end{equation}
With a rearrangement argument via Lemma \ref{Lemma}, we can assume that $(\phi_n)$ is radial decreasing and satisfies \eqref{suite}.\\
$\bullet$ First step: $(\phi_n)$ is bounded in $H.$\\
First subcase $\alpha\neq 0.$ Write
{\small\begin{gather*}
\alpha \Big(\displaystyle \sum_{j=1}^m \|\phi_j^n\|_{H^1}^2  - \displaystyle\sum_{j,k=1}^m a_{jk}\displaystyle\int_{\R^N}|\phi_j^n \phi_k^n|^p\,dx\Big) = \frac{\beta}{2}\Big( 2 \displaystyle \sum_{j=1}^m \|\nabla \phi_j^n\|^2 - N \displaystyle\sum_{j=1}^m \|\phi_j^n\|_{H^1}^2
 + \frac{N}{p} \displaystyle\sum_{j,k=1}^m a_{jk}\displaystyle \int_{\R^N}|\phi_j^n\phi_k^n|^p\,dx\Big);\\
\displaystyle\sum_{j=1}^m  \| \phi_j^n\|_{H^1}^2   -  \frac{1}{p}\displaystyle\sum_{j,k=1}^m a_{jk} \displaystyle\int_{\R^N}|\phi_j^n \phi_k^n|^{p}\,dx \rightarrow 2m.
\end{gather*}}
Denoting $\lambda:= \frac{\beta}{2\alpha},$ yields
{\small$$\displaystyle \sum_{j=1}^m \|\phi_j^n\|_{H^1}^2  - \displaystyle\sum_{j,k=1}^m a_{jk}\displaystyle\int_{\R^N}|\phi_j^n \phi_k^n|^p\,dx = \lambda\Big( 2 \displaystyle \sum_{j=1}^m \|\nabla \phi_j^n\|^2 - N \displaystyle\sum_{j=1}^m \|\phi_j^n\|_{H^1}^2+ \frac{N}{p} \displaystyle\sum_{j,k=1}^m a_{jk}\displaystyle \int_{\R^N}|\phi_j^n\phi_k^n|^p\,dx\Big). $$}
So the following sequences are bounded
\begin{gather*}
- 2 \lambda \displaystyle\sum_{j=1}^m\|\nabla\phi_j^n\|^2 + \displaystyle \sum_{j=1}^m \|\phi_j^n\|_{H^1}^2  - \displaystyle\sum_{j,k=1}^m a_{jk}\displaystyle\int_{\R^N}|\phi_j^n \phi_k^n|^p\,dx;\\
\displaystyle\sum_{j=1}^m \|\phi_j^n\|_{H^1}^2 - \frac{1}{p} \displaystyle\sum_{j,k=1}^m a_{jk}\displaystyle \int_{\R^N}|\phi_j^n\phi_k^n|^p\,dx. 
\end{gather*}
Thus, for any real number $a,$ the following sequence is also bounded
$$ 2 \lambda \displaystyle\sum_{j=1}^m\|\nabla \phi_j^n\|^2  + (a - 1) \displaystyle\sum_{j=1}^m \|\phi_j^n\|_{H^1}^2 + (1 - \frac{a}{p}) \displaystyle\sum_{j,k=1}^m a_{jk}\displaystyle \int_{\R^N}|\phi_j^n\phi_k^n|^p\,dx.$$
Choosing $a\in(1,p),$ it follows that $(\phi_n)$ is bounded in $H.$\\
Second subcase $\alpha=0$ and $N\geq3.$ Write
$$\sum_{j=1}^m \|\nabla \phi_j^n\|^2\lesssim H_{0,\beta}({\bf \phi^n})=\frac{1}{2\alpha +N\beta }\Big[\displaystyle\sum_{j=1}^m 2\beta \|\nabla\phi_j^n\|^2  + \alpha (1-\frac{1}{p})\displaystyle\sum_{j,k=1}^m a_{jk}\displaystyle\int_{\R^N}|\phi_j^n\phi_k^n|^{p}\,dx\Big]\leq m.$$
Assume that $\displaystyle\lim_n\displaystyle\sum_{j=1}^m \| \phi_j^n\|=\infty$. Then, taking account of the interpolation inequality \eqref{Nirenberg}, we get
$$\sum_{j=1}^m \| \phi_j^n\|^2\lesssim K^Q({\bf \phi^n})=(\alpha+\frac{N\beta}{2p})\displaystyle\Sigma_{1\leq j,k\leq m}\int_{\R^N}|u_ju_k|^p\,dx\lesssim \left(\displaystyle\sum_{j=1}^{m}\|\phi_j^n\|^2\right)^{\frac{N-p(N -2)}{2}}.$$
This is a contradiction because $\frac{N-p(N -2)}{2}=p-\frac{N(p-1)}2<1$.\\
$\bullet$ Second step: the limit of $(\phi_n)$ is nonzero and $m>0.$\\
Taking account of the compact injection \eqref{radial}, we take
$$ (\phi_1^n,...,\phi_m^n) \rightharpoonup \phi= (\phi_1,...,\phi_m)\quad \mbox{in}\quad H$$
and
$$ (\phi_1^n,...,\phi_m^n) \rightarrow  (\phi_1,...,\phi_m) \quad\mbox{in}\quad (L^{2p})^{(m)}.$$
The equality $K(\phi_n) =0$ implies that
$$  \frac{2\alpha + (N-2)\beta}{2} \displaystyle\sum_{j=1}^m \|\nabla\phi_j^n\|^2 + \frac{2\alpha + N\beta}{2} \displaystyle\sum_{j=1}^m \|\phi_j^n\|^2 =\frac{2\alpha p + N\beta}{2p} \displaystyle\sum_{j,k=1}^m a_{jk}\displaystyle \int_{\R^N} |\phi_j^n \phi_k^n|^p\,dx.$$
Assume that $\phi =0$. Using H\"older inequality 
$$\|\phi_j^n\phi_k^n\|_p^p \leq \|\phi_j^n\|_{2p}^p \|\phi_k^n\|_{2p}^p \rightarrow \|\phi_j\|_{2p}^p \|\phi_k\|_{2p}^p = 0.$$
Now, by lemma \ref{K>0} yields $K(\phi_n)>0$ for large $n$. This contradiction implies that 
$$\phi \neq 0.$$
With lower semi continuity of the $H^1$ norm, we have 
\begin{eqnarray*}
0 &=& \liminf_n K(\phi_n)\\
&\geq& \frac{2\alpha + (N-4)\beta}{2} \liminf_n\displaystyle\sum_{j=1}^m \|\nabla \phi_j^n\|^2 + \frac{2\alpha + N\beta}{2} \liminf_n\displaystyle\sum_{j=1}^m \|\phi_j^n\|^2\\
& -&\frac{2\alpha p + N\beta}{2p} \displaystyle\sum_{j,k=1}^m a_{jk}\displaystyle \int_{\R^N} |\phi_j \phi_k|^p\,dx\\
&\geq& K(\phi).
\end{eqnarray*}
Similarly, we have $H(\phi)\leq m.$ Moreover, thanks to Lemma \ref{Lemma}, we can assume that $K(\phi) =0$ and $S(\phi) = H(\phi)\leq m.$ So that $\phi$ is a minimizer satisfying \eqref{suite} and
$$m_{\alpha,\beta}=H(\phi) = \frac{1}{2\alpha + N\beta}\Big( 2\beta\displaystyle\sum_{j=1}^m \|\nabla\phi_j\|^2  +  \alpha(1 - \frac{1}{p})\displaystyle\sum_{j,k=1}^m a_{jk}\displaystyle\int_{\R^N} |\phi_j\phi_k|^p\,dx\Big)>0.$$
$\bullet$ Third step: the limit $\phi$ is a solution to \eqref{E}.\\
There is a Lagrange multiplier $\eta \in \R$ such that $S^\prime(\phi) = \eta K^\prime(\phi).$ Thus
$$0 = K(\phi) = \pounds S(\phi)= \langle S^\prime(\phi), \pounds(\phi)\rangle = \eta\langle K^\prime(\phi), \pounds(\phi)\rangle = \eta\pounds K(\phi) = \eta\pounds^2S(\phi). $$
With a previous computation, we have
{\small\begin{eqnarray*}
-\pounds^2 S(\phi) - (2\alpha+(N-2)\beta)(2\alpha + N\beta)S(\phi)& =& - (\pounds - (2\alpha+(N-2)\beta))(\pounds - (2\alpha + N\beta))S(\phi) \\
&=&\frac{1}{2p}2\alpha(p - 1)(2\alpha(p -1) + 2\beta)\displaystyle\sum_{j,k=1}^m a_{jk}\displaystyle\int_{\R^N}|\phi_j\phi_k|^p\,dx\\
&>0.&
\end{eqnarray*}}
Therefore $\pounds^2S(\phi)<0.$
Thus $\eta =0$ and $S^\prime(\phi) = 0.$ So, $\phi$ is a ground state and $m$ is independent of $\alpha,\, \beta.$\\
\underline{Second case $(\alpha,N)=(0,2)$.}\\
With a rearrangement argument, take $(\phi_n):=(\phi_1^n,...,\phi_m^n) $ be a radial minimizing sequence, satisfying \eqref{suite}.\\
$\bullet$ First step: $(\phi_n)$ is bounded in $H.$\\
Without loss of generality, take $\beta=1.$ We have
\begin{gather*}
 H_{0,1}(\phi_n) = \displaystyle\sum_{j=1}^m \|\nabla\phi_j^n\|^2;\\
 K_{0,1}(\phi_n) = \displaystyle\sum_{j=1}^m\|\phi_j^n\|^2 - \frac{1}{p}\displaystyle\sum_{j,k=1}^m a_{jk}\displaystyle \int_{\R^N} |\phi_j^n\phi_k^n|^p\,dx.\end{gather*}
By \eqref{suite} via the definition of $H_{0,1},$ $(\phi_n)$ is bounded in $(\dot H^1)^{(m)}.$ Now, because
$$ H_{0,1}(\phi_n)= H_{0,1}(\phi_n^\lambda),\quad K_{0,1}(\phi_n^\lambda) = e^{2\lambda}K_{0,1}(\phi_n) =0,$$ 
by the scaling $\phi_n^\lambda:= \phi_n(e^{-\lambda}.),$ we may assume that $\|\phi^n\| =1$ for $j \in [1,m].$ Thus $(\phi_n) $ is bounded in $ H.$\\
$\bullet$ Second step: the limit of $(\phi_n)$ is nonzero and $m>0.$\\
Taking account of the compact injection $\eqref{radial},$ we take
$$ (\phi_1^n,...,\phi_m^n) \rightharpoonup (\phi_1,...,\phi_m)\quad \mbox{in}\quad H$$
and
$$ (\phi_1^n,...,\phi_m^n) \rightarrow (\phi_1,...,\phi_m) \quad\mbox{in}\quad (L^{2p})^{(m)}.$$
Now, by the fact $0 = K (\phi_n),$ we have
$$  1 = \frac{1}{p}\displaystyle\sum_{j,k=1}^m\displaystyle\int_{\R^4} a_{jk}|\phi_j^n\phi_k^n|^p\,dx .$$
Moreover, if $\phi=0,$ we get
$$\|\phi_j^n\phi_k^n\|_p^p\leq \|\phi_j^n\|_{2p}^p\|\phi_k^n\|_{2p}^p\to \|\phi_j\|_{2p}^p\|\phi_k\|_{2p}^p =0,$$
which is a contradiction. Then
$$\phi \neq 0.$$
For $0<\lambda\longrightarrow0,$ we have 
$$ \displaystyle\sum_{j,k=1}^m a_{jk}\displaystyle\int |\lambda \phi_j|^p |\lambda \phi_k|^p\,dx = o(\displaystyle\Sigma_{j=1}^mK^Q(\lambda \phi_j)) = \lambda^2\displaystyle\Sigma_{j=1}^m K^Q(\lambda \phi_j).$$
 Thus,
$$ K_{0,1}(\phi)<0\Rightarrow \exists \lambda\in(0,1),\;\mbox{s.\,t}\; K_{0,1}(\lambda \phi)=0\,\,\mbox{and}\,\, H_{0,1}(\lambda \phi)\leq H_{0,1}(\phi).$$
So, we may assume that $K(\phi) =0$ and $S(\phi) = H(\phi)\leq m.$ Then $\phi$ is a minimizer and $m = H(\phi)>0.$\\
$\bullet$ Third step: The limit $\phi$ is a solution to \eqref{E}.\\
With a lagrange multiplicator $\eta\in \R,$ we have $S^\prime(\phi) = \eta K^\prime(\phi).$ Moreover, since
$$S^\prime(\phi_j) =-\Delta \phi_j + \phi_j - \displaystyle\sum_{k=1}^m a_{jk}|\phi_k|^p|\phi_j|^{p-2}\phi_j\quad\mbox{and}\quad K^\prime(\phi_j) =2 \phi_j - 2\displaystyle\sum_{k=1}^m a_{jk} |\phi_k|^p|\phi_j|^{p-2}\phi_j. $$
it follows that
$$ \Delta \phi_j = (1-2\eta)\big(\phi_j - \displaystyle\sum_{k=1}^m|\phi_k|^p|\phi_j|^{p -2}\phi_j\big).$$
Since $\langle \Delta \phi_j,\phi_j\rangle<0$ and 
\begin{eqnarray*}
 \displaystyle\sum_{j=1}^m\displaystyle\int_{\R^N}\big(|\phi_j|^2 - \displaystyle\sum_{k=1}^m |\phi_k\phi_j|^p\big)\,dx
&=&  K_{0,1}(\phi)  -\displaystyle\sum_{j=1}^m \|\phi_j\| ^2 + (\frac1{p} - 1) \displaystyle\sum_{j,k=1}^m a_{jk}\displaystyle\int_{\R^N}|\phi_k\phi_j|^p\,dx \\
&=&  -\displaystyle\sum_{j=1}^m \|\phi_j\| ^2 + (\frac1{p} - 1) \displaystyle\sum_{j,k=1}^m a_{jk}\displaystyle\int_{\R^N}|\phi_k\phi_j|^p\,dx                    <0.
\end{eqnarray*}
Then $(1-2\eta)>0.$ Finally, choosing a real number $\lambda$ such that $e^{-2\lambda}(1-2\eta) =1$, existence of a ground state follows taking account of the equality
$$ \Delta (\phi_j(e^{- \lambda}.)) = e^{-2\lambda}(1-2\eta)\Big(\phi_j(e^{-\lambda}.)-\displaystyle\sum_{k=1}^m | \phi_k(e^{-\lambda}.)|^p |\phi_j(e^{-\lambda}.)|^{p - 2}\phi_j(e^{-\lambda}.)   \Big).$$
\underline{Third case $(\alpha,\beta) = (1,-\frac2N).$}\\
 Let $(\phi_n):=(\phi^n_1,..,\phi^n_m)$ a minimizing sequence, namely
\begin{equation}\label{m'}
0\neq\phi_n\in H,\;K_{1,-\frac2N}(\phi_n)= 0\;\mbox{and}\;\lim_nS(\phi_n)=m.
\end{equation}
With a rearrangement argument, we can assume that $\phi_n$ is radial decreasing and satisfies
$$0\neq\phi_n\in H,\;K_{1,-\frac2N}(\phi_n)\leq 0\;\mbox{and}\;\lim_nS(\phi_n)\leq m.$$
We can suppose that $\phi_n$ is radial decreasing and satisfies \eqref{m'}. Indeed, by Lemmas \ref{cle}-\ref{dec}, there exists $\lambda\in(0,1)$ such that $K_{1,-\frac2N}(\lambda\phi_n)=0$ and $T(\lambda\phi_n)\leq m$. Then
\begin{gather*}
\frac2N\displaystyle\displaystyle\Sigma_{j=1}^m\|\nabla\phi_j\|^2=(1-\frac1p)\displaystyle\Sigma_{j,k}a_{j,k}\int_{\R^N}|\phi_j\phi_k|^p\,dx;\\
\displaystyle\displaystyle\Sigma_{j=1}^m\|\phi_j\|_{H^1}^2-\frac1p\displaystyle\Sigma_{j,k}a_{j,k}\int_{\R^N}|\phi_j\phi_k|^p\,dx\rightarrow 2m.
\end{gather*}
So, for any real number $a\neq 0$,
{\small$$\displaystyle\displaystyle\Sigma_{j=1}^m\Big[(1-\frac{2a}N)\|\nabla\phi_j\|^2+\|\phi_j\|^2+(a-\frac{1+a}p)\displaystyle\displaystyle\Sigma_{k=1}^ma_{j,k}\int_{\R^N}|\phi_j\phi_k|^p\Big]\,dx\rightarrow 2m.$$}
Letting $a\in(\frac1{p-1},\frac N2)$ gives that $(\phi_n)$ is bounded in $H$. Taking account of the compact injection $\eqref{radial},$ we take
$$ (\phi_1^n,...,\phi_m^n) \rightharpoonup (\phi_1,...,\phi_m)\quad \mbox{in}\quad H$$
and
$$ (\phi_1^n,...,\phi_m^n) \rightarrow (\phi_1,...,\phi_m) \quad\mbox{in}\quad (L^{2p})^{(m)}.$$
Assume, by contradiction, that $\phi=0$. The equality $K(\phi_n) =0$ implies, via H\"older inequality and the fact that $p_*<p<p^*,$
\begin{eqnarray*}
\frac2N\displaystyle\displaystyle\Sigma_{j=1}^m\|\nabla u_j\|^2
&=&(1-\frac1p)\displaystyle\Sigma_{j,k}a_{j,k}\|\phi_j^n\phi_k^n\|_p^p\\
&\lesssim& \displaystyle\Sigma_{j,k}\|\phi_j^n\|_{2p}^p \|\phi_k^n\|_{2p}^p\\
& \rightarrow&\displaystyle\Sigma_{j,k} \|\phi_j\|_{2p}^p \|\phi_k\|_{2p}^p = 0.
\end{eqnarray*}
Now, by lemma \ref{K>0} yields $K(\phi_n)>0$ for large $n$. This contradiction implies that 
$$\phi \neq 0.$$
Thanks to lower semi continuity of $H^1$ norm, we have $K_{1,-\frac2N}(\phi)\leq0$ and $S(\phi)\leq m$. Using Lemmas \ref{dec}-\ref{cle}, we can assume that
$K_{1,-\frac2N}(\phi)=0$ and $T(\phi)\leq m_{1,-\frac2N}.$ So $\phi$ is a minimizer satisfying
$$0\neq\phi\in H_{rd},\quad K_{1,-\frac2N}(\phi)=0\quad\mbox{ and }\quad T(\phi)=m.$$
This implies that
 $$0<\|\phi\|^2\leq T(\phi)=m.$$
Now, there is a Lagrange multiplier $\eta\in\R$ such that $S'(\phi)=\eta K_{1,-\frac2N}'(\phi)$. Thus
\begin{eqnarray*}
0=K_{1,-\frac2N}(\phi)&=&\mathcal L_{1,-\frac2N}S(\phi)=\langle S'(\phi),\mathcal L_{1,-\frac2N}(\phi)\rangle\\
&=&\eta\langle K_{1,-\frac2N}'(\phi),\mathcal L_{1,-\frac2N}(\phi)\rangle\\
&=&\eta\mathcal L_{1,-\frac2N}K_{1,-\frac2N}(\phi)=\eta\mathcal L_{1,-\frac2N}^2S(\phi).\end{eqnarray*}
With a direct computation, we have $\mathcal L_{1,-\frac2N}(\|\phi_j\|^2)=(\mathcal L_{1,-\frac2N}-\frac4N)(\|\nabla\phi_j\|^2)=0$ and $\mathcal L_{1,-\frac2N}(|\phi_j\phi_k|^p)=2(p-1)|\phi_j\phi_k|^p$. So
\begin{eqnarray*}
-\mathcal L_{1,-\frac2N}(\mathcal L_{1,-\frac2N}-\frac4N)S_j(\phi)
&=&(1-\frac1p)(p-1-\frac2N)\displaystyle\Sigma_{j}\int_{\R^N}|u_ju_k|^p\;dx>0.
\end{eqnarray*}
Then, $-\mathcal L_{1,-\frac2N}^2S(\phi)>0$, so $\eta=0$ and $S'(\phi)=0$. Finally, $\phi$ is a ground state solution to \eqref{E}.\\
\subsection{Existence of vector ground state}
Now, we present a proof of the last part of Theorem \ref{t1}, which deals with existence of a more that one non zero component ground state for large $\mu.$ Take $\phi: =(\phi_1,...,\phi_m)$ such that $(0,...,\phi_{j},...,0)$ is a ground state solution to \eqref{E}. So, $\phi_{j}$ satisfies
$$\Delta  \phi_j - \phi_j + \mu_j \phi_j|{\phi_j}|^{2p-2}=0\quad\mbox{and}\quad\displaystyle\sum_{j=1}^m\|\phi_{j}\|_{H^1}^2 =\displaystyle\sum_{j=1}^m \mu_j\|\phi_{j}\|_{2p}^{2p}.$$
Moreover, by Pohozaev identity it follows that
$$\frac{N - 2}{2}\displaystyle\sum_{j=1}^m\|\nabla \phi_{j}\|^2 + \frac{N}{2}\displaystyle\sum_{j=1}^m\|\phi_{j}\|^2 = \frac{N}{2p}\displaystyle\sum_{j=1}^m \mu_j\|\phi_{j}\|_{2p}^{2p}.$$
Collecting the previous identities, we can write
\begin{equation}\label{phz}
\displaystyle\sum_{j=1}^{m}\|\phi_{j}\|^2 = \Big(1 - \frac{N}{2} + \frac{N}{2p}\Big)\displaystyle\sum_{j=1}^{m} \mu_j\|\phi_{j}\|_{2p}^{2p}.
\end{equation}
Setting, for $t>0,$ the real variable function $\gamma(t):=\big(\phi_{1}(\frac{.}{t}),...,\phi_{m}(\frac{.}{t})\big),$ compute
\begin{eqnarray*}
K_{0,1}(\gamma(t)) &=& \frac{N - 2}{2 }t^{N - 2}\displaystyle\sum_{j=1}^m\|\nabla\phi_j\|^2  + \frac{N }{2}t^{ N}\displaystyle\sum_{j=1}^m\|\phi_j\|^2 - \frac{N}{2p}t^{ N}\displaystyle\sum_{j=1}^m \mu_j\|\phi_j\|_{2p}^{2p} \\ & -& \frac{N}{2p} \mu t^{ N}\displaystyle\sum_{ 1 \leq k\neq j\leq m} \displaystyle\int_{\R^N}|\phi_j\phi_k|^p\,dx
\end{eqnarray*}
and
\begin{eqnarray*}
g(t)
&:=&S(\gamma(t))\\
 &=& \frac{1}{2}t^{N - 2}\displaystyle\sum_{j=1}^m\|\nabla\phi_j\|^2  +  \frac{1}{2}t^{ N}\displaystyle\sum_{j=1}^m\|\phi_j\|^2 - \frac{1}{2p}t^{ N}\displaystyle\sum_{j=1}^m \mu_j\|\phi_j\|_{2p}^{2p} \\
 & -& \frac{1}{2p} \mu t^{ N}\displaystyle\sum_{1\leq k\neq j\leq m} \displaystyle\int_{\R^N}|\phi_j\phi_k|^p\,dx.
\end{eqnarray*} 
Thanks to \eqref{phz}, $ g(t)<0$ for large $t.$ Then, since $g(0)\geq0$ The maximum of $g(t)$ for $t\geq0$ is achieved at $\bar t>0.$ Precisely $g(\bar t) = \displaystyle\max_{t\geq0} g(t).$ Moreover, 
\begin{eqnarray*}
g^{\prime}(\bar t)=0&=& {\bar t}^{N -1}\Big(\frac{N -2}{2} {\bar t}^{-2}\displaystyle\sum_{j=1}^m\|\nabla\phi_j\|^2 + \frac{N}{2}\displaystyle\sum_{j=1}^m\|\phi_j\|^2 - \frac{N}{2p}\displaystyle\sum_{j=1}^m \mu_j\|\phi_j\|_{2p}^{2p} \\&-& \frac{N}{2p}\mu\displaystyle\sum_{1\leq j\neq k\leq m}\displaystyle\int_{\R^N}|\phi_j\phi_k|^p\,dx\Big).
\end{eqnarray*} 
Then,
$${\bar t}= \Big(\frac{N - 2}{N}\Big)^{\frac{1}{2}}\frac{\left(\displaystyle\sum_{j=1}^m\|\nabla\phi_j\|^2\right)^{\frac{1}{2}}}{\left( \frac{1}{p}\displaystyle\sum_{j=1}^m \mu_j\|\phi_j\|_{2p}^{2p}   +  \frac{1}{p}\mu \displaystyle\sum_{1\leq j\neq k\leq m} \displaystyle\int_{\R^N}|\phi_j\phi_k|^p\,dx - \displaystyle\sum_{j=1}^m \|\phi_j\|^2\right)^{\frac{1}{2}}}.$$
Thus, the maximum value of $g$ is 
\begin{eqnarray*}
 g(\bar t)& =& \displaystyle\max_{t\geq0} g(t)\\
 &=&{\frac{2(N -2)^{\frac{N -2}{2}}}{N^{\frac{N}{2}}}}\frac{\left(\displaystyle\sum_{j=1}^m\|\nabla\phi_j\|^2\right)^{\frac{N}{2}}}{\left( \frac{1}{p}\displaystyle\sum_{j=1}^m \mu_j\|\phi_j\|_{2p}^{2p}   +  \frac{1}{p}\mu \displaystyle\sum_{1\leq j\neq k\leq m} \displaystyle\int_{\R^N}|\phi_j\phi_k|^p\,dx - \displaystyle\sum_{j=1}^m \|\phi_j\|^2\right)^{\frac{N - 2}{2}}}. 
\end{eqnarray*}
Now, take $\bar{{\bf u}}:=(\bar{u}_1,..,\bar{u}_m)$ a ground state to \eqref{S}, when $\mu\longrightarrow\infty$, from the previous equality via the fact that $K_{0,1}(\gamma(t))<0$ for large $\mu$, it follows that
$$0<m=S(\bar  {{\bf u}})\leq S(\phi_1(\frac.{\bar t}),...,\phi_m(\frac.{\bar t}))\longrightarrow0.$$
This contradiction achieves the proof.
\section{Invariant sets and applications}
This section is devoted to obtain global and non global existence of solutions to the system \eqref{S}. Precisely, we prove Theorem \ref{t2}. We start with a classical result about stable sets under the flow of \eqref{S}. 
\begin{lem}\label{lem}The sets $A_{\alpha,\beta}^+$ and $A_{\alpha,\beta}^-$ are invariant under the flow of \eqref{S}.
\end{lem}
\begin{proof} Let $ \Psi \in A_{\alpha,\beta}^+$ and ${\bf u} \in C_{T^*}(H)$ be the maximal solution to \eqref{S}. Assume that ${\bf u}(t_0) \not \in A_{\alpha,\beta}^+$ for some $t_0\in (0,T^*)$. Since $S({\bf u})$ is conserved, we have $K_{\alpha,\beta}({\bf u}(t_0))<0.$ So, with a continuity argument, there exists a positive time $t_1\in (0, t_0)$ such that $ K_{\alpha,\beta}({\bf u}(t_1)) = 0$ and $S({\bf u}(t_1))<m.$ This contradicts the definition of $m.$ The proof is similar in the case of $A_{\alpha,\beta}^-$.
\end{proof}{}
The previous stable sets are independent of the parameter $(\alpha,\beta)$.
\begin{lem}\label{fn}
The sets $A_{\alpha,\beta}^+$ and $A_{\alpha,\beta}^-$ are independent of $(\alpha,\beta)$.
\end{lem}
\begin{proof}
Let $(\alpha,\beta)$ and $(\alpha',\beta')$ in ${\mathbb R}_+^2\cup{(1,-\frac2N)}-\{(0,0)\}$. We denote, for $\delta\geq0$, the sets 
\begin{gather*}
A_{\alpha,\beta}^{+\delta}:=\{{\bf u}\in H\quad\mbox{s.\, t}\quad S({\bf u})<m-\delta\quad\mbox{and}\quad K_{\alpha,\beta}({\bf u})\geq0\};\\
A_{\alpha,\beta}^{-\delta}:=\{{\bf u}\in H\quad\mbox{s.\, t}\quad S({\bf u})<m-\delta\quad\mbox{and}\quad K_{\alpha,\beta}({\bf u})<0\}.
\end{gather*}
By Theorem \ref{t1}, the reunion $A_{\alpha,\beta}^{+\delta}\cup A_{\alpha,\beta}^{-\delta}$ is independent of $(\alpha,\beta)$. So, it is sufficient to prove that $A_{\alpha,\beta}^{+\delta}$ is independent of $(\alpha,\beta)$.The rescaling ${\bf u}^\lambda:=e^{\alpha\lambda}{\bf u}(e^{-\beta\lambda}.)$ implies that a neighborhood of zero is in $A_{\alpha,\beta}^{+\delta}$. If $S({\bf u})<m$ and $K_{\alpha,\beta}({\bf u})=0$, then $ {\bf u}=0$. So, $A_{\alpha,\beta}^{+\delta}$ is open. Moreover, this rescaling with $\lambda\rightarrow-\infty$ gives that $A_{\alpha,\beta}^{+\delta}$ is contracted to zero and so it is connected. Now, write $$A_{\alpha,\beta}^{+\delta}=A_{\alpha,\beta}^{+\delta}\cap( A_{\alpha',\beta'}^{+\delta}\cup A_{\alpha',\beta'}^{-\delta})=(A_{\alpha,\beta}^{+\delta}\cap A_{\alpha',\beta'}^{+\delta})\cup(A_{\alpha,\beta}^{+\delta}\cap A_{\alpha',\beta'}^{-\delta}).$$ 
Since by the definition, $A_{\alpha,\beta}^{-\delta}$ is open and $0\in A_{\alpha,\beta}^{+\delta}\cap A_{\alpha',\beta'}^{+\delta}$, using a connectivity argument, we have $A_{\alpha,\beta}^{+\delta}=A_{\alpha',\beta'}^{+\delta}$.
\end{proof}
\subsection{Global existence}
With a translation argument, we assume that $t_0 = 0.$ Thus, $S(\Psi)<m$ and with lemma \ref{lem}, ${\bf u}(t)\in A_{\alpha,\beta}^+$ for any $t\in [0, T^*).$ Moreover,
\begin{eqnarray*}
m&\geq& \big(S -\frac{1}{2 + N} K_{1,1}\big) ({\bf u})\\
&=& H_{1,1}({\bf u})\\
&=&\frac{1}{2 + N}\Big( 2\displaystyle\sum_{j=1}^m\|\nabla u_j\|^2  + (1 - \frac{1}{p})\displaystyle\sum_{j,k=1}^ma_{jk}\displaystyle\int_{\R^N}|u_ju_k|^p\,dx \Big)\\
&\geq& \frac{2}{2 + N}\displaystyle\sum_{j=1}^m\|\nabla u_j\|^2.
\end{eqnarray*}
Thus, ${\bf u}$ is bounded in (\.H$^1)^{(m)}.$ Precisely
$$ \sup_{0\leq t\leq T^*}\displaystyle\sum_{j=1}^m\|\nabla u_j\|^2\leq \frac{(2 + N)m}{2}.$$
Moreover, since the $L^2$ norm is conserved, we have
$$ \sup_{0\leq t\leq T^*}\displaystyle\sum_{j=1}^m\| u_j\|_{H^1}^2 <\infty.$$
Finally, $T^* = \infty.$
\subsection{Non global existence}
With a translation argument, we assume that $t_0=0$. Thus, $S({\bf u})<m$ and with Lemma \ref{lem}, ${\bf u}(t)\in A_{\alpha,\beta}^-$ for any $t\in[0,T^*)$. By contradiction, assume that $T^*=\infty$. Take the real function $Q(t):=\sum_{j=1}^m\int_{\R^N}|x|^2|u_j(t)|^2\,dx$. With \eqref{vrl}, we get
$$\frac18Q''(t)=\frac2NK_{1,-\frac{2}N}({\bf u(t)}).$$
We infer that there exists $\delta>0$ such that $K_{1,-\frac{2}N}({\bf u(t)})<-\delta$ for large time. Otherwise, there exists a sequence of positive real numbers $t_n\rightarrow+\infty$ such that $K_{1,\frac{-2}N}({\bf u}(t_n))\rightarrow0.$ By Lemma \ref{Lemma}, yields
$$m\leq(S-K_{1,-\frac{2}N})({\bf u}(t_n))=S({\bf u}(0,.))-K_{1,-\frac{2}N}({\bf u}(t_n))\rightarrow S({\bf u}(0,.)))<m.$$
This absurdity finishes the proof of the claim. Thus $Q''<-8\delta$. Integrating twice, $Q$ becomes negative for some positive time. This contradiction closes the proof.
\section{Strong instability}
This section is devoted to prove Theorem \ref{t3} about strong instability of standing waves.
The next intermediate result reads
\begin{lem}
Let $\Psi$ to be a ground state solution of \eqref{E}, $\lambda>1$ a real number close to one and ${\bf u}_\lambda$ the solution to \eqref{S} with data $\Psi_{\lambda}:=\lambda^{\frac{N}2}\Psi(\lambda.)$. Then, for any $t\in(0,T^*)$,
$$S({\bf u}_\lambda(t))<S(\Psi)\quad\mbox{and}\quad K_{1,-\frac2N}({\bf u}_\lambda(t))<0.$$
\end{lem}
\begin{proof}
By Lemma \ref{cle}, we have  
$$S(\Psi_{\lambda})<S(\Psi)\quad\mbox{and}\quad K_{1,-\frac2N}(\Psi_{\lambda})<0.$$ 
Moreover, thanks to the conservation laws, it follows that for any $t>0$, 
$$S({\bf u}_\lambda(t))=S(\Psi_{\lambda}(t))<S(\Psi).$$
Then $K_{1,-\frac2N}({\bf u}_\lambda(t))\neq0$ because $\Psi$ is a ground state. Finally $K_{1,-\frac2N}({\bf u}_\lambda(t))<0$ with a continuity argument.
\end{proof}
Now, we are ready to prove the instability result.\\
{\bf Proof of Theorem \ref{t3}}.
Take ${\bf u}_{\lambda}\in C_{T^*}(H)$ the maximal solution to \eqref{S} with data $\Psi_{\lambda}$, where $\lambda>1$ is close to one and $\Psi$ is a ground state solution to \eqref{E}. With the previous Lemma, we get 
$${\bf u}_\lambda(t)\in A_{1,-\frac2N}^-, \quad\mbox{for any}\quad t\in(0,T^*).$$
Then, using Theorem \ref{t2}, it follows that
$$\lim_{t\rightarrow T^*}\|{\bf u}_\lambda(t)\|_H=\infty.$$
The proof is finished via the fact that 
$$\lim_{\lambda\rightarrow1}\|\Psi_\lambda-\Psi \|_H=0.$$


\end{document}